\documentclass[11pt]{article}
\usepackage{
amsmath,
amsfonts,
latexsym, 
amsrefs,
amssymb,
mathtools,
tikz,
tabularx,
minibox,
amsthm,
}
\usepackage[toc,page]{appendix}
\usepackage{hyperref}
\usepackage{tocloft}
\usepackage[noadjust]{cite}
\usepackage{graphicx}
\usepackage{bm}
\usepackage{dsfont}
\usepackage{stmaryrd}
\usepackage{enumerate}
\usepackage{mathrsfs}
\usepackage{wrapfig}
\usepackage{color}

\usetikzlibrary{shapes}
\usetikzlibrary{intersections}

\numberwithin{equation}{section}
\newcommand{\bla}{\bm{\lambda}}

\newcommand{\La}{\Lambda}

\newcommand{\bu}{{\mathbf u}}
\newcommand{\bv}{{\mathbf v}}

\newcommand{\de}{\delta}
\newcommand{\be}{\begin{equation}}
\newcommand{\ee}{\end{equation}}

\newcommand{\ga}{{\gamma}}

\newcommand{\la}{\lambda}
\newcommand{\barla}{\overline{\lambda}}

\newcommand{\D}{\mathbb{D}}

\newcommand{\Pf}{\rm{Pf}}
\newcommand{\VGamma}{\Gamma_{V}}
\newcommand{\Spec}{\rm{Spec}}

\newcommand{\E}{\mathbb{E}}
\newcommand{\R}{\mathbb{R}}
\newcommand{\C}{\mathbb{C}}
\renewcommand{\H}{\mathbb{H}}

\newcommand{\dd}{\mathrm{d}}
\newcommand{\tr}{\mathrm{tr}}

\newcommand{\Ov}{\mathscr{O}}
\newcommand{\disteq}{\stackrel{d}{=}}
\newcommand{\distconv}{\xrightarrow[N \rightarrow \infty]{d}}
\newcommand{\rr}{\right}
\renewcommand{\ll}{\left}

\theoremstyle{plain} 
\newtheorem{theorem}{Theorem}[section]
\newtheorem*{theorem*}{Theorem}
\newtheorem{lemma}[theorem]{Lemma}
\newtheorem*{lemma*}{Lemma}
\newtheorem{corollary}[theorem]{Corollary}
\newtheorem*{corollary*}{Corollary}
\newtheorem{proposition}[theorem]{Proposition}
\newtheorem*{proposition*}{Proposition}
\newtheorem{fact}[theorem]{Fact}
\newtheorem*{fact*}{Fact}

\newtheorem{definition}[theorem]{Definition}
\newtheorem*{definition*}{Definition}

\newtheorem*{example*}{Example}

\newtheorem*{remark*}{Remark}
\newtheorem*{remarks*}{Remarks}

\makeatletter
\renewcommand{\section}{\@startsection
{section}
{1}
{0mm}
{-2\baselineskip}
{1\baselineskip}
{\normalfont\large\scshape\centering}} 
\makeatother

\makeatletter
\renewcommand{\subsection}{\@startsection
{subsection}
{2}
{0mm}
{-\baselineskip}
{1 \baselineskip}
{\normalfont\scshape}} 
\makeatother

\renewcommand{\rm}{\mathrm}

\def\@empty{}

\def\author#1{\par
    {\centering{\authorfont#1}\par\vspace*{0.05in}}}

\def\titlefont{\fontsize{12}{15} \centering{}}
\def\authorfont{\fontsize{12}{15}}

\let\affiliationfont\rhfont

\def\address#1{\par
    {\centering{\affiliationfont#1\par}}\par\vspace*{11pt}
}

\def\keywords#1{\par
    \vspace*{8pt}
    {\authorfont{\leftskip18pt\rightskip\leftskip
    \noindent{\it\small{Keywords}}\/:\ #1\par}}\vskip-12pt}

\def\body{
\setcounter{footnote}{0}
\def\thefootnote{\alph{footnote}}
\def\@makefnmark{{$^{\rm \@thefnmark}$}}
}

\def\title#1{ 
    \thispagestyle{plain}
    \vspace*{-14pt}
    \vskip 79pt
    {\centering{\titlefont #1\par}}%
    \vskip 1em
}

\begin{document}

~\vspace{-2cm}

\title{\textsc{Symmetries of the Quaternionic Ginibre Ensemble}}

\vspace{0.3cm}
\author{Guillaume Dubach}
\address{Courant Institute, New York University \\
dubach@cims.nyu.edu}
\vspace{0.3cm}

\begin{abstract}
We establish a few properties of eigenvalues and eigenvectors of the quaternionic Ginibre ensemble (QGE), analogous to what is known in the complex Ginibre case (see \cite{BourgadeDubach, Dubach1, Fyodorov}). We first recover a version of Kostlan's theorem that was already noticed by Rider \cite{RiderQGE}: the set of the squared radii of the eigenvalues is distributed as a set of independent gamma variables. Our proof technique uses the De Bruijn identity and properties of Pfaffians; it also allows to prove that the high powers of these eigenvalues are independent. These results extend to any potential beyond the Gaussian case, as long as radial symmetry holds; this includes for instance truncations of quaternionic unitary matrices, products of quaternionic Ginibre matrices, and the quaternionic spherical ensemble. \medskip

We then study the eigenvectors of quaternionic Ginibre matrices. The angle between eigenvectors and the matrix of overlaps both exhibit some specific features that can be compared to the complex case. In particular, we compute the distribution and the limit of the diagonal overlap associated to an eigenvalue that is conditioned to be zero.
\end{abstract}

\vspace{.5cm}

\tableofcontents
\newpage

\section{Introduction}

In his 1965 seminal paper \cite{Ginibre}, Jean Ginibre studied three ensembles of Gaussian random matrices now known to us as the real, complex and quaternionic Ginibre ensembles. He did so by order of complexity: firstly, the complex case, for which he could derive all correlation functions and give a first proof of convergence to the circular law; secondly the quaternionic case, for which he was able to compute the joint density of eigenvalues. The real Ginibre ensemble, though in a way the most natural of the three, has in fact a less tractable structure. Accordingly, the purpose of this note is to establish, for the quaternionic Ginibre ensemble (QGE), a version of what is known for its complex counterpart, but not yet for the real one. \medskip

\subsection{Quaternions and quaternionic matrices}

Although we first recall a few facts about quaternions, the reader should be aware that the techniques used in Section \ref{EigenvaluesSection} only rely on the knowledge of the joint distribution of eigenvalues (\ref{density}). In the same way, the results concerning eigenvectors in Section \ref{EigenvectorsSection} only rely on the distribution of the Schur transform, recalled as Theorem \ref{QuatSch}. \medskip

Hamilton's quaternions (also called real quaternions, as opposed to complex quaternions or biquaternions) form a noncommutative division algebra $\H$ whose elements can be written
$$
q = a+bi +cj +dk
$$
where $a,b,c,d \in \R$ and $i,j,k$ follow the relations 
$$
i^2 = j^2 = k^2 = -1, \quad
ij=-ji=k, \quad
jk=-kj=i, \quad
ki=-ik=j, 
$$
the real numbers being the center of the algebra. If $A \in \mathscr{M}_{N}(\H)$ is a matrix of quaternions, we say that $\lambda \in \H$ is an eigenvalue of $A$ if there exists a vector $X \in \H^N \backslash \{{\bf 0}\}$ such that
$$
AX=X \lambda
$$
where the side chosen for multiplication is of great importance, as it ensures in particular that $\lambda^k$ is an eigenvalue of $A^k$. We will denote $\Spec (A)$ the set of such eigenvalues.

\begin{fact}
If $\lambda \in \Spec (A)$, then for any $q \in \H^*$, $q^{-1} \lambda q \in \Spec (A)$.
\end{fact}

The proof of this fact is elementary: if $X$ is an eigenvector for $A$, then $Xq$ is an eigenvector for $q^{-1} \lambda q$, as
$$
A (Xq) = (X \lambda)q = X q (q^{-1} \lambda q),
$$
where we only used associativity of quaternions. This makes it clear that, rather than considering eigenvalues, it is more appropriate to consider classes of eigenvalues. We denote by $C_\la$ the conjugacy class of $\la$ in $\H$.

\begin{fact}
If $\la \in \R$, then $C_{\la}=\{ \la \}$. If $\la \notin \R$, then $C_{\la}$ is a two-dimensional sphere.
\end{fact}

The proof of this fact is a straightforward computation. If $\la = a + b i+ cj + dk$, then denoting $r^2=b^2+c^2+d^2$,
$$
C_{\la} =\{ a + \beta i+ \gamma j + \delta k \ | \ \beta^2 + \gamma^2 + \delta^2 = r^2 \} = S(a,r).
$$ 
Note that, although this sphere is centered on the real axis, it doesn't intersect it, unless it is reduced to one point. This is no paradox, as we are considering a surface of dimension $2$ in a space of dimension $4$. In the same way, eigenvectors associated to a sphere of eigenvalues lie in a subspace isomorphic to $\H$. The relevance of this fact will appear more clearly when we consider 'angles' between eigenvectors in Subsection \ref{AngleSection}. \medskip

Other general facts about quaternions, quaternionic matrices, quaternionic hermitianity and quaternionic determinants can be found in the literature (see for instance \cite{Aslaksen, Lee}). We will now focus on the quaternionic random matrix ensemble first considered by Ginibre.

\subsection{Quaternionic Ginibre Ensemble}

Let $G$ be a matrix from the quaternionic Ginibre ensemble (QGE). This means that every quaternionic coefficient $q_{i,j}$ of the matrix is chosen with i.i.d. real Gaussian coordinates $ a_{i,j}, b_{i,j}, c_{i,j}, d_{i,j}$. With probability $1$, $\Spec (G)$ consists of $N$ spheres that do not intersect the real axis. Nevertheless, we can speak of this spectrum as a point process in $\C$: the reason is that every such sphere intersects the complex plane twice, and these two points are complex conjugates. \medskip

As the point processes that we consider in this note are symmetric with respect to the real axis, it will be convenient to denote by $\Psi$ the conjugate duplication of a set of points, namely, for any set $A$ of complex numbers,
$$ \Psi(A) = A \cup \bar{A}.$$

We also define the Gaussian reference measure as
$$\dd \mu (z) = \frac{1}{\pi} e^{-|z|^2} \dd m(z)$$
where $\dd m$ is the Lebesgue measure on $\C$. We shall also refer to this as the standard complex Gaussian, $\mathscr{N}_{\C}(0,1)$. \medskip

With these notations, the eigenvalues of the quaternionic Ginibre ensemble, first computed in \cite{Ginibre}, are given by $\Psi \left( \{ \lambda_i\}_{i=1}^N \right)$, where the distribution of $\bla= (\lambda_1, \dots, \lambda_N) \in \mathbb{C}^N$ is given by
\begin{equation}\label{density}
\frac{1}{Z_N} \prod_{i<j} |\la_i-\la_j|^2 \prod_{i\leq j} |\la_i-\barla_j|^2 \dd \mu^{N}(\bla)
\end{equation}
with $Z_N = 2^N N! \prod_{i=1}^N (2i-1)!$. All eigenvalues are distinct with probability one. 
Convergence to the circular law after scaling by $\sqrt{2N}$ has been proved in \cite{Benaych}. We will refer to the eigenvalues given by the density $(\ref{density})$ as 'unscaled'. For some results however, it will make more sense to consider properly scaled eigenvalues.

\subsection{Old and new results}

All results established here for the quaternionic Ginibre ensemble have a known analog in the complex Ginibre case. This is summarized in the synoptic table below, with comments on the next page.

\vspace{0.4cm}

{
\begin{tabular}{|l|c|c|}  
  \cline{2-3}
\multicolumn{1}{c|}{} &  &  \\
\multicolumn{1}{c|}{} & Complex Ginibre Ensemble  & Quaternionic Ginibre Ensemble \\
\multicolumn{1}{c|}{} &  &   \\
  \hline
 &  &  \\
  & Independence & Independence \\
 Squared Radii & $\{ |\lambda_k|^2 \}_{k=1}^N \stackrel{d}{=} \{ \gamma_k \} _{k=1}^N$ & $ \{ |\lambda_k|^2 \}_{k=1}^N \stackrel{d}{=} \{ \gamma_{2k} \} _{k=1}^N $ \\
    &  &  \\
  & (Kostlan \cite{Kost}) & (Rider \cite{RiderQGE} and Theorem \ref{KostlanQuat}) \\
   &  &  \\
 \hline
  &  &  \\
  & Independence for $M \geq N$ & Independence for $M \geq 2N$ \\
High Powers & $\{ \lambda_k^M \}_{k=1}^N \stackrel{d}{=} \{ Z_k \} _{k=1}^N$ & $ \Psi \left( \{ \lambda_k^M \}_{k=1}^N \right) \stackrel{d}{=} \Psi\left( \{ Z_k \} _{k=1}^N \right)$  \\
  & with $(Z_k)_{k=1}^N$ independent. & with $(Z_k)_{k=1}^N$ independent. \\
     &  &  \\
    & (Hough \& al. \cite{HKPV})  & (Theorem \ref{HighPowers}) \\
   &  &  \\
   \hline
      &  &  \\
Angle between & $ \frac{\langle R_1 | R_2 \rangle }{\| R_1\| \| R_2 \|} \disteq \phi\ll( \frac{X}{\la_1 - \la_2} ,0 \rr)$ & $ \frac{\langle R_1 | R_2 \rangle }{\| R_1\| \| R_2 \|}  \disteq \phi\left( \frac{X}{\la_1 - \la_2} ,\frac{Y}{\barla_1 - \la_2} \right)$ \\
  eigenvectors    & where $X \sim \mathscr{N}_{\C}(0,1)$ & where $X,Y$ are i.i.d. $\mathscr{N}_{\C}(0,1)$ \\
   &  &  \\
      & (\cite{BenZeitouni}, Appendix B in \cite{BourgadeDubach}) & (Proposition \ref{Angledistrib}) \\
      &  &  \\
   \hline
      &  &  \\
Distribution of & $\Ov_{1,1} \disteq \prod_{k=2}^N \ll( 1+ \frac{|X_k|^2}{|\la_1 - \la_k|^2} \rr) $ & $\Ov_{1,1} \disteq \prod_{k=2}^N \ll( 1+ \frac{|X_k|^2}{|\la_1 - \la_k|^2} + \frac{|Y_k|^2}{|\la_1 - \barla_k|^2} \rr) $ \\
   diagonal overlaps     & $X_k$ are i.i.d. $\mathscr{N}_{\C}(0,1)$ & $X_k, Y_k$ are i.i.d. $\mathscr{N}_{\C}(0,1)$ \\
& (Theorem 2.2 in \cite{BourgadeDubach}) & (Theorem \ref{diagoverlap})  \\
      &  &  \\
      \hline
       &  &  \\
 Limit of overlaps  & $\frac{1}{N} \Ov_{1,1} \distconv \ga_2^{-1}$ & $\frac{1}{N} \Ov_{1,1} \distconv ( \frac{1}{2} \ga_4)^{-1}$ \\
   at $\la_1=0$   & (Proposition 2.4 in \cite{BourgadeDubach}) & (Theorem \ref{Gamma4ever})  \\
     &  &  \\
       &  &  \\
   \hline
\end{tabular}
}

\newpage

\begin{itemize}

\item The first result, Theorem \ref{KostlanQuat}, is the analog of Kostlan's theorem: the square radii of the eigenvalues are distributed, as a set of independent (but not i.i.d.) gamma variables. The same is true for QGE, with different parameters. These radii of pairs of complex numbers correspond to the distance to the origin for the whole spheres of quaternionic eigenvalues. \medskip

Note that, for this theorem as well as the next one, it is not technically true to say that the concerned statistics are independent; what we prove is the identity between two point processes, one of them involving independent points that are not identically distributed.  \medskip

\item The second result, Theorem \ref{HighPowers} is that high powers of QGE, as well as high powers of complex Ginibre eigenvalues, become distributed as a set of independent variables. This is not an asymptotic result, but an exact decomposition, as soon as the power reaches a certain level (corresponding, in both cases, to the number of points involved). \medskip

\item The third section of the paper is devoted to properties of eigenvectors, accessible through the structure of the Schur transform, which is triangular but has essentially the same eigenstatistics as the original quaternionic matrix. The first natural question that can be asked is how the angle between two given eigenvectors is distributed. It is also straightforward to answer, as this angle depends only on the first few coefficients of the Schur transform. The function $\phi$ that appears in the synoptic table maps pairs of complex numbers to the open unit disk in the following way:
\begin{align*}
\phi: \quad \C^2 & \rightarrow \D \\
 (z,w) & \mapsto \frac{z}{\sqrt{1+|z|^2+|w|^2}}.
\end{align*}
The essential feature of the distribution of angles is that, as in the complex case, eigenvectors corresponding to eigenvalues that are at least mesoscopically separated tend to be close to orthogonal.

\item The last two rows of the synoptic table concern the diagonal elements of the matrix of overlaps, defined in the beginning of Subsection \ref{OverlapSection}. Overlaps are quantities involving both left and right eigenvectors of the matrix, that turn out to be more relevant, in many cases, than the angles between eigenvectors. In a sense, they quantify the non-normality of the matrix, as well as its sensitivity to perturbations. For this reason, diagonal overlaps are sometimes referred to as the eigenvalue condition numbers. The distribution and correlations of overlaps in general are hard to compute, but one useful fact is that, for QGE as well as for complex Ginibre, the diagonal overlap is distributed like a product of independent terms involving i.i.d. Gaussian variables and the eigenvalues. This equality in distribution, together with Theorem \ref{KostlanQuat}, leads to a limit theorem when conditioning on $\{ \la_1 = 0\}$. Extending this result to other values of $\la_1$ is a natural question that we leave aside for now.

\end{itemize}
\newpage
\section{Eigenvalues of QGE}\label{EigenvaluesSection}

In this section, after briefly recalling the properties of Pfaffians, we establish two results for the eigenvalues of the quaternionic Ginibre ensemble, analogous to what is know for the complex Ginibre ensemble. 

\subsection{Pfaffians and De Bruijn formula}

We define the Pfaffian of a $2N \times 2N$ matrix A as
$$
\Pf A = \frac{1}{2^N N!} \sum_{\sigma \in \mathfrak{S}_{2N}} \epsilon(\sigma) \prod_{i=1}^N A_{\sigma(2i),\sigma(2i-1)}
$$

Note that the definition holds, whether $A$ is skew-symmetric or not, but one can always bring it down to the the Pfaffian of a skew-symmetric matrix by using the equality

\begin{fact}\label{skewsym} For any $A \in \mathscr{M}_{2N}$,
$$
\Pf (A) = \Pf \left( \frac{A-A^T}{2} \right).
$$
\end{fact}
We will also use the following elementary facts.
\begin{fact}\label{basics} For any $A,B \in \mathscr{M}_{2N}(\C), M \in \mathscr{M}_{N}(\C)$,
$$
\Pf(BAB^t) = \det (B) \Pf (A), \qquad
\Pf
\left(
\begin{array}{cc}
0 & M \\
-M^t & 0
\end{array}
\right)
= (-1)^{\frac{N(N-1)}{2}} \det M.
$$
\end{fact}

The following non trivial proposition is the primary interest of the Pfaffian. 

\begin{proposition}\label{Muir} If $A$ is skew-symmetric, then 
$$
(\Pf A)^2=\det A
$$
\end{proposition}

The next identity, first established in \cite{DeBruijn}, is essential for our purpose, as it will play the same role here as did Andreief's identity in \cite{Dubach1}.

\begin{proposition}[De Bruijn]\label{DeBru}
Let $(E, \mathcal{E}, \nu)$ be a measure space. For any functions $(\phi_i, \psi_i)_{i=1}^{2N} \in L^2(\nu)^{4N}$,
$$
\int_{E^N} \det \left(  \phi_i(\lambda_j) \ | \ \psi_i (\lambda_{j}) \right) \ \dd \nu^{  N} ({\bla}) = N! 2^N \Pf \left(f_{i,j}\right)_{i,j=1}^{2N}, \quad \text{where} \ f_{i,j} = \int_E  \phi_i(\lambda) \psi_j(\lambda) \dd \nu( \lambda).
$$
\end{proposition}
The bar in the determinant only means that we define the $2N \times 2N$ matrix in two halves, according to columns, so that the index $i$ goes from $1$ to $2N$ once, and the index $j$ goes from $1$ to $N$ twice. \medskip

The results quoted above are well-known. Proofs and comments can be found in \cite{DeBruijn, Deift1}. De Bruijn's identity gives a formula for the product statistics of QGE.

\begin{corollary}[Product Statistics]\label{ProStat} Let $g \in L^2(\mu)$ such that $g(\lambda)=g(\bar{\lambda})$, and  $\Psi \left( \{ \lambda_1, \dots, \lambda_N\} \right)$ be unscaled quaternionic Ginibre eigenvalues. Then,
$$ \E \left( \prod_{k=1}^N g \left(\lambda_k \right) \right)= \frac{1}{\prod_{i=1}^N (2i-1)!} \Pf \left( f_{i,j} \right)_{i,j = 1}^{2N} \quad \text{where} \ f_{i,j} = \int_{\mathbb{C}} \left(z^{i} \bar{z}^{j-1} - z^{i-1} \bar{z}^{j} \right) g(z) \dd \mu ( z).$$ 
\end{corollary}
\begin{proof} We first note that the density (\ref{density}) can essentially be written as one $2N \times 2N$ Vandermonde determinant. Indeed,
$$
\prod_{i<j} |z_i-z_j|^2 \prod_{i\leq j} |z_i-\bar{z_j}|^2
=
\det \left( z_i^{j-1} \ | \ \overline{z_i}^{j-1} \right)
\prod_{i=1}^N (\overline{z}_i - z_i).
$$
Using De Bruijn formula \ref{DeBru} with $E=\mathbb{C}$, $\phi_j(z)=z^{i-1}$, $\psi_j(z)=\bar{z}^{j-1}$  and the complex measure $\dd \nu(z) = (\bar{z}-z) g(z) \dd \mu(z)$ yields the result.
\end{proof}

It is known that, for any suitable reference measure beyond the Gaussian case, such statistics where $g$ is a polynomial in $\la, \barla$ characterize the distribution of the point process. A detailed argument is provided in \cite{Dubach1}.

\subsection{Radii of QGE}

The square radii of these complex eigenvalues are relevant to the quaternionic eigenvalues: they are the radii of the spheres of eigenvalues in $\mathbb{H}$. \medskip

The following theorem is the analog of Kostlan's independence theorem in the complex case. It was proved by Rider in \cite{RiderQGE} with different techniques, and used to establish a Gumbel limit for the fluctuations of the largest radius of the eigenvalues. The same could be done with a more general reference measure, assuming some regularity, using the methods of \cite{ChafaiPeche}.

\begin{theorem}[Rider]\label{KostlanQuat} If $\Psi \left( \{ \lambda_i \}_{i=1}^N \right)$ is the (unscaled) spectrum of $G$, then
$$\{ |\lambda_i|^2 \}_{i=1}^N \stackrel{d}{=} \{ X_i \}_{i=1}^N$$
where the $X_i$ are independent, and $X_i \sim \gamma(2i)$.
\end{theorem}

\begin{proof}
Let $g \in \mathbb{C}[X]$ and apply Corollary \ref{ProStat} to the radially symmetric function $g(|\cdot|^2)$. The matrix is then tridiagonal, with coefficients
\begin{align*}
f_{i,i} & = 0 \\
f_{i,i+1}& = \int |z|^{2i} g(|z|^2) \dd \mu ( z) = \int_{r=0}^{\infty} r^{i} g(r) e^{-r} \dd r = i! \ \E \left( g(\gamma_{i+1}) \right)\\
f_{i+1,i} & = - f_{i,i+1}.
\end{align*}
As the matrix $(f_{i,j})$ is skew-symmetric, we know by Proposition \ref{Muir} that its Pfaffian is the square root of its determinant. We have 
$$
\E \left( \prod_{i=1}^N g(|\la_i|^2) \right) =
\frac{1}{\prod_{i=1}^N (2i-1)!}
\Pf(f_{i,j})
=
\frac{1}{\prod_{i=1}^N (2i-1)!}
\sqrt{\det(f_{i,j})}.
$$
It is straightforward to see that the determinant of such a tridiagonal matrix is the product of the squares of every other upper-diagonal term. 
$$
\det
\left(
\begin{array}{cccccc}
0     & b_1  & 0      & \dots   &   0 \\
-b_1  & 0    & b_2    & \ddots  & \vdots \\
0     &-b_2  & \ddots & \ddots  & 0 \\
\vdots&\ddots& \ddots & 0       & b_{2N-1} \vspace{2mm}  \\
0     &\dots &   0    &-b_{2N-1}& 0 
\end{array}
\right)
 = \prod_{i=1}^N b_{2i-1}^2
$$

That is to say,
$$
\E \left( \prod_{i=1}^N g(|\la_i|^2) \right)
=
\frac{1}{\prod_{i=1}^N (2i-1)!}
\sqrt{\left(\prod_{i=1}^N (2i-1)! \E\left( g(\gamma_{2i} ) \right) \right)^2}
=
\E \left(  \prod_{i=1}^N g(X_i) \right). 
$$
These statistics characterize the distribution of a set of points, as such expressions with polynomial $g$ generate all symmetric polynomials (see the appendix of \cite{Dubach1}), and the distributions involved are characterized by their moments. We conclude that $ \{ |\lambda_1|^2, \dots, |\lambda_N|^2 \} \stackrel{d}{=} \{ \gamma_2, \dots, \gamma_{2N}\}$.
\end{proof}

In particular, it follows that, if the largest eigenvalue is to have order $1$, the appropriate scaling is $\sqrt{2N}$. The same method allows us to obtain the following lemma:

\begin{lemma}\label{Kostlancond} Conditionally on $\{ \la_1=0 \}$, then
$$\{ |\lambda_i|^2 \}_{i=2}^N \stackrel{d}{=} \{ X_i \}_{i=2}^N$$
where the $X_i$ are independent, and $X_i \sim \gamma(2i)$.
\end{lemma}

We shall use this conditional result in the proof of Theorem \ref{Gamma4ever}.

\subsection{High Powers of QGE}

The same technique yields the following result for high powers (namely, any power greater or equal to the number of eigenvalues). This is the analog of a general property of determinantal point processes with radial symmetry, proved in \cite{HKPV}. The first result of this kind was obtained by Rains \cite{Rainspower} for the CUE.

\begin{theorem}\label{HighPowers} For any integer $M \geq 2N$, the following equality in distribution holds:
$$
\Psi \left(
\{ \lambda_1^M, \dots, \lambda_N^M \} 
\right)
\stackrel{d}{=}
\Psi \left(
\{\gamma_2^{M/2} e^{i \theta_1}, \dots, \gamma_{2N}^{M/2} e^{i \theta_N} \}
\right)
$$
where the variables $\gamma_{2k},\theta_k$ are independent, the gamma variables having parameters $2,4, \dots, 2N$, and the angles being uniform on $[0,2\pi]$.
\end{theorem}

\begin{proof}
We use Corollary \ref{ProStat} with a polynomial $P(\la^M, \bar{\la}^M)$. This gives
$$
f_{i,j}= \int_{\mathbb{C}} (z^i \bar{z}^{j-1} - z^{i-1} \bar{z}^{j}) P(z^M,\bar{z}^M) \dd \mu (z)
$$
Because of radial symmetry, the only terms that do not vanish are those for which $i-j \equiv \pm 1 [M]$. But as $M \geq 2N$ and the matrix size is $2N$, this only happens when $i=j+1$ or $j=i+1$, and we have a skew-symmetric tridiagonal matrix of the same kind as in Theorem \ref{KostlanQuat}, so that
$$
\Pf (f_{i,j})
= \sqrt{\det (f_{i,j})}
= \prod_{i=1}^N f_{2i-1,2i}
$$
with
$$
f_{2i-1,2i}
=
\int_{\mathbb{C}} |z|^{4i-2} P(z^M,\bar{z}^M) \dd \mu (z)
$$
and finally
$$
\E \left( \prod_{i=1}^N P(\la_i^M,\bar{\la}_i^M) \right)
=
\frac{1}{\prod_{i=1}^N (2i-1)!}
\prod_{i=1}^N f_{2i-1,2i}
= \E \left( \prod_{i=1}^N P(Z_i^M,\bar{Z}_i^M) \right)
$$
where the variables $Z_i$ are distributed as claimed.
\end{proof}

\subsection{Extension to a general Potential}
Instead of the complex Gaussian measure $\dd \mu$, the reference measure could be given by any external potential, provided it has radial symmetry. This is done in the exact same way as in \cite{Dubach1}. All results stated here could be stated more generally with respect to
$$\dd \mu_V(z) = \frac{1}{Z_V} e^{- V \left( |z|^2 \right)} \dd m ( z),$$ 
where the function $V: \R_+ \mapsto \R \cup \{ \infty \}$ is such that

\begin{equation}\label{finiteness}
\int_{\C^N} \prod_{1 \leq i<j \leq N} |z_i-z_j|^2 \prod_{1 \leq i \leq j \leq N} |z_i- \bar{z_j}|^2 \dd \mu_V^{  N} (\mathbf{z}) \ < \ \infty.
\end{equation}
so that, when properly normalized, this defines a probability density. Note that, while we sometimes call $V$ the potential, strictly speaking the potential is given by $V(|z|^2)$, such that the quadratic potential case corresponds to $V= \rm{Id_{\R_+}}$.  
\begin{definition}\label{Vgammafun}
We denote by $\VGamma$ the analog of the $\Gamma$ function with potential $V$,
$$\VGamma(\alpha) = \int_{0}^{\infty} t^{\alpha-1} e^{-V(t)} \dd t.$$
As long as $\alpha$ is such that the above is integrable, we define the $\VGamma$ distribution of parameter $\alpha$, denoted by $\gamma(V,\alpha)$, by its density on $\mathbb{R}_+$, \begin{equation}\label{Vgamma}
\frac{1}{\VGamma(\alpha)} t^{\alpha-1} e^{-V(t)} \mathds{1}_{\mathbb{R}_+}.
\end{equation}
\end{definition}
In this more general setting, Theorem \ref{KostlanQuat} holds replacing the usual gamma variables by gamma-$V$ variables, and high powers ($M \geq 2N$) still exhibit the same kind of independence. Forrester \cite{Forrester} gives the joint eigenvalue densities of three ensembles that fall into this category.


\begin{itemize}
\item{\bf Products of quaternionic Ginibre matrices.} \\ The eigenvalue distribution of $Y=G_1 \dots G_k$ where $G_1, \dots, G_k$ are independent Ginibre matrices is given by 
$$
\prod_{i=1}^N w_k^{(4)}(|z_i|^2)
\prod_{1 \leq i<j \leq N} |z_i-z_j|^2 \prod_{1 \leq i \leq j \leq N} |z_i- \bar{z_j}|^2.
$$
where weight $w_k^{(4)}$ is defined and studied in \cite{AkemannIpsen, Ipsenthesis, Ipsenquaternion}. It is closely related to the weight $w_k^{(2)}$ appearing in the complex Ginibre case, and thus to Meijer G-functions.


\item {\bf Truncated Quaternionic Unitary Matrices.} \\  
These matrices are equivalent to symplectic unitary matrices. Minors of size $N$ from quaternionic unitary matrices of size $N+n$ (distributed according to the Haar measure) have eigenvalue density proportional to
$$
\prod_{k=1}^N (1-|z_k|^2)^{2n-1} \mathbf{1}_{|z_k|<1} 
\prod_{1 \leq i<j \leq N} |z_i-z_j|^2 \prod_{1 \leq i \leq j \leq N} |z_i- \bar{z_j}|^2.
$$
In that case, the $\VGamma$ variables are usual beta variables. Namely, the set of radii is distributed as a set of independent variables with distribution $\beta_{2,2n}, \beta_{4,2n} \dots, \beta_{2N,2n}$.


\item {\bf Quaternionic Spherical ensemble.} \\
This ensemble corresponds to the distribution of eigenvalues of $G_1^{-1} G_2$ where $G_1, G_2$ are i.i.d. quaternionic Ginibre matrices of size $N$. The eigenvalue density is then proportional to
$$
\prod_{k=1}^N \frac{1}{(1+|z_k|^2)^{2(N+1)}}
\prod_{1 \leq i<j \leq N} |z_i-z_j|^2 \prod_{1 \leq i \leq j \leq N} |z_i- \bar{z_j}|^2.
$$
In that case, variables are more or less heavy-tailed, and cannot be characterized by their moments. The results still holds, using the same procedure as in \cite{Dubach1} to extend them.
\end{itemize}

\newpage
\section{Eigenvectors of QGE}\label{EigenvectorsSection}

It is important to clarify what we mean by an eigenvector in the quaternionic case. For this last section, we will use the traditional representation of real quaternions as $2 \times 2$ matrices, namely
$$
a+ bi + cj + dk = \ll(
\begin{array}{cc}
a+bi & c+di \\
-c+di & a-bi
\end{array}
\rr)= \ll(
\begin{array}{cc}
z & - \overline{w} \\
w & \overline{z}
\end{array}
\rr).
$$
Consequently, quaternionic matrices of size $N \times N$ are now considered as usual complex matrices of size $2N \times 2N$, made of $2 \times 2$ blocks that represent quaternions. \medskip

For every matrix with distinct eigenvalues, one can consider a biorthogonal system of left and right eigenvectors. Indeed, writing the matrix as $P \Delta P^{-1}$ with $\Delta$ diagonal, and calling $\langle L_i | $ the rows of $P^{-1}$ and $ | R_j \rangle $ the columns of $P$, we see that these are right and left eigenvectors such that 
\begin{equation}\label{biorthogonality}
L_i R_j = \langle L_i | R_j \rangle = \de_{i,j}.
\end{equation}
All eigenvalues being distinct and non real with probability one, we will often assume implicitly that eigenvectors are distinct and well-defined as biorthogonal projective quantities. This however will not be true when we condition on $\{ \la_1 = 0 \}$, but we will then fix one reference eigenvector. \medskip

We denote by $\langle \cdot  | \cdot \rangle$ the complex (sesquilinear) scalar product. On the other hand, the bilinear product will be denoted by a dot, so that if $\bu, \bv \in \C^d$ are column vectors, $\langle \bu | \bv \rangle = \bu^* .\bv $. \medskip

It will be useful to define the following map on complex vectors of even length $2d$,
\begin{align*}
\Phi: \quad & \C^{2d} \rightarrow \C^{2d} \\
0 \leq k \leq d-1, \quad & u_{2k+1} \mapsto - \overline{u_{2k+2}} \\
& u_{2k+2} \mapsto \overline{u_{2k+1}}
\end{align*}
and to notice the following facts: $\Phi$ is an involution, with
\begin{equation}\label{factsphi}
\forall \bu, \bv \in \C^{2d} \quad
\bu.\Phi(\bv) = - \overline{\Phi(\bu).\bv} \qquad \text{and}
\quad
\langle {\bf u} | \Phi({\bf u}) \rangle = 0.
\end{equation}

Subsection \ref{AngleSection} below states a structure theorem about the Schur transform of QGE, and illustrates it by studying the distribution of the angle between eigenvectors. \medskip

Subsection \ref{OverlapSection} studies a more complex object, namely the matrix of eigenvector overlaps. Overlaps are homogeneous quantities involving both left and right eigenvectors. \medskip

By convention, $G$ will stand for the unscaled QGE matrix, but we will make it clear in every result whether the eigenvalues we consider are the scaled or the unscaled ones.

\subsection{Angle between eigenvectors}\label{AngleSection}
The Schur decomposition of quaternionic Ginibre matrices is described as follows. \medskip

\begin{proposition}[Schur decomposition]\label{QuatSch}
The quaternionic matrix $G$, written under complex form, is unitarily equivalent to the following upper-triangular matrix:
$$
\ll(
\begin{array}{cccc}
\La_1 & T_{1,2} & \dots & T_{1,N} \\
0 & \La_2 & \ddots & \vdots \\
\vdots & \ddots & \ddots & T_{N-1,N} \\
0 & \dots & 0 & \La_N
\end{array}
\rr)
$$
where all terms are $2 \times 2$ blocks, namely
$
\La_i = \ll(
\begin{array}{cc}
\la_i & 0 \\
0 & \barla_i
\end{array}
\rr)
$
and
$
T_{i,j}= \ll(
\begin{array}{cc}
u_{i,j} & v_{i,j} \\
-v_{i,j}^* & u_{i,j}^*
\end{array}
\rr).
$ \medskip

Moreover, diagonal blocks are independent from upper-diagonal ones, and the $T_{i,j}$ blocks are i.i.d. and distributed with $u_{i,j}$, $v_{i,j}$ standard complex Gaussian variables.
\end{proposition}

See for example \cite{Ipsenthesis} for a proof. One divergence in the notations is that the Gaussian measure we consider corresponds to the weight
$$
e^{-\frac{1}{2} \tr M^*M}.
$$
instead of $e^{- \tr M^*M}$. Note that if we denote by $T_{(k)}$ the $k$th column of the Schur form, we have, for any $d$, the following quaternionic identity:
\begin{equation}\label{Schurcolumns}
T_{(2d+2)} = \Phi(T_{(2d+1)}).
\end{equation}

We define the 'complex angle' between the right-eigenvectors associated to $\la_1$ and $\la_2$ as
$$
\rm{arg(\la_1, \la_2)}:= \frac{\langle R_1 | R_2 \rangle }{\| R_1 \| \| R_2 \|}.
$$
It is clear that, by exchangeability, all pairs involving distinct eigenvalues are equivalent -- only the pairs involving an eigenvalue and its own conjugate play a different role. It is also clear, by unitary equivalence, that the statistics of the angle between eigenvectors are the same for a random matrix and its Schur form. \medskip

A straightforward computation, using biorthogonality and the definition of left and right eigenvectors, shows that, if we call $R_i,\widetilde{R}_i$ the eigenvectors associated to $\la_i$ and $\barla_i$ respectively,
$$
\begin{array}{ll}
R_1^* =(1, 0, 0, \dots, 0)
& \quad L_1=(1,0,b_3, \dots, b_{2N}) \\
\widetilde{R}_1^* = (0,1,0, \dots, 0)
& \quad \widetilde{L}_1 = (0,1, c_3, \dots, c_{2N}) \\
R_2^* =(-b_3, -c_3 , 1 , 0 , 0 \dots, 0)
& \quad L_2=(0,0,1,0, d_5, \dots, d_{2N}) \\
\widetilde{R}_2^* =(-b_4, -c_4, 0, 1, 0 \dots, 0)
& \quad \widetilde{L}_2 =(0,0,0,1,e_5, \dots, e_{2N}) 
\end{array}
$$
where the coefficients $(b_i,c_i, d_i, e_i)$ are defined by two-term recursions. We only write down the first one, the others being similar {\it mutatis mutandis}:

\begin{equation}\label{recursion}
\ll\{
\begin{array}{rl}
b_{2d+1} & = \frac{1}{\la_1 - \la_d} (b^{(2d)}.T_{(2d)}) \\
b_{2d+2} & = \frac{1}{\la_1 - \barla_d} (b^{(2d+1)}.T_{(2d+1)})
= \frac{1}{\la_1 - \barla_d} (b^{(2d)}.\Phi(T_{(2d)})) \\
\end{array}
\rr.
\end{equation}
where $b^{(k)}$ stands for a column vector containing the first $k$ entries of $L_1$. \medskip

We define the following function, that maps the complex plane to the open unit disk.
\begin{align*}
\phi: \quad \C^2 & \rightarrow \D \\
 (z,w) & \mapsto \frac{z}{\sqrt{1+|z|^2+|w|^2}}
\end{align*}
The following proposition is stated with respect to scaled eigenvalues, i.e. eigenvalues of $\frac{1}{\sqrt{2N}} G$.

\begin{proposition}[Distribution of angle between eigenvectors]\label{Angledistrib} Conditionally on $(\la_1, \la_2) \in \D^2$, we have the following identities:
$$
\arg(\la_1, \barla_1) =  \arg(\la_2, \barla_2) = 0,
$$
\vspace{-.1cm}
$$
\arg(\la_1,\la_2) = \phi\left( \frac{X}{\sqrt{2N}(\la_1 - \la_2)} ,\frac{Y}{\sqrt{2N}(\barla_1 - \la_2)} \right), \quad
\arg(\barla_1,\barla_2) = \overline{\arg(\la_1,\la_2)},
$$
$$
\arg(\barla_1,\la_2) = \phi\left( \frac{-Y^*}{\sqrt{2N}(\barla_1 - \la_2)}, \frac{X}{\sqrt{2N}(\la_1 - \la_2)}  \right), \quad
\arg(\la_1,\barla_2) = - \overline{\arg(\barla_1,\la_2)},
$$

where $X,Y$ are i.i.d. standard complex Gaussian variables.
\end{proposition}

This result indicates that, similarly to eigenvalues forming spheres of quaternions, we should speak of subspaces of eigenvectors. With the complex notations, the relevant object associated to the pair $(\la_i, \barla_i)$ is the plane generated by the orthogonal vectors $R_i, \widetilde{R}_i$. One direct consequence of Proposition \ref{Angledistrib} is that one expects these planes to be asymptotically orthogonal when the associated eigenvalues are macroscopically (or at least mesoscopically) separated.

\begin{proof} It is clear from the definition of $\arg$ and what has been computed above, that $\arg(\la_1, \barla_1) = 0$, and therefore by exchangeability every analog quantity is zero as well. The other pairs yield:
$$
\arg(\la_1, \la_2) = \frac{-\overline{b}_3}{\sqrt{1+|b_3|^2 + |c_3|^2}},
\qquad
\arg(\la_1, \barla_2) = \frac{-\overline{b}_4}{\sqrt{1+|b_4|^2 + |c_4|^2}}
$$
$$
\arg(\barla_1, \la_2) = \frac{-\overline{c}_3}{\sqrt{1+|b_3|^2 + |c_3|^2}},
\qquad
\arg(\barla_1, \barla_2) = \frac{-\overline{c}_4}{\sqrt{1+|b_4|^2 + |c_4|^2}}
$$\
Using the structure of $T_{1,2}$ from Proposition \ref{QuatSch} and the recursion formulae (\ref{recursion}), we find that
$$
b_3 = \frac{u_{1,2}}{\la_1-\la_2}, \quad b_4 = \frac{v_{1,2}}{\la_1-\barla_2},
\quad
c_3 = \frac{-v_{1,2}^*}{\barla_1-\la_2}, \quad c_4 = \frac{u_{1,2}^*}{\barla_1-\barla_2}
$$
which concludes the proof if we set $X=- \overline{u}_{1,2}\sqrt{N}, Y= - \overline{v}_{1,2} \sqrt{N}$.
\end{proof}

We will now state a result that holds conditionally on $\{ \la_1 = 0\}$. What we mean is that we define the eigenvector associated to $\la_1$ to be $R_1=(1, 0, 0, \dots, 0)$. Note that the choice of $\widetilde{R}_1$ instead would lead to the same result. This result as well as the ones stated in Subsection \ref{OverlapSection} can be motivated even though eigenvalues exhibit repulsion with their own inverse, which at every finite $N$ makes the presence of an eigenvalue in the vicinity of the real line unlikely: we know that with the appropriate scaling, the limit is the circular law, which density does not vanish around the real line. So the following proposition is to be asymptotically expected for eigenvalues that are microscopically close to $0$ when $N$ goes to infinity.

\begin{proposition} Conditionally on $\{ \la_1 = 0\}$, the following identity in distribution holds:
$$
|\arg(\la_1, \la_2)|^2
\disteq
\beta_{1,I+1}
$$
where $I$ is a random variable uniform on $\{ 4, 6, \dots 2N \}$, and $\beta_{a,b}$ stands for a beta distribution with parameters $a,b$.
\end{proposition}

\begin{proof}
Proposition \ref{AngleSection}, Theorem \ref{KostlanQuat} and general properties of beta and gamma variables yield
$$
|\arg(\la_1, \la_i)|^2
\disteq
\frac{|X|^2}{|\la_i|^2+|Y|^2+|X|^2}
\disteq
\frac{\ga_1}{\ga_{I} + \widetilde{\ga}_1 + \ga_1}
\disteq
\beta_{1,I}
$$
which concludes the proof.
\end{proof}

\subsection{The matrix of overlaps}\label{OverlapSection}

We define the overlaps between the eigenvectors associated to two eigenvalues $\la_i$ and $\la_j$ in the following way~:
$$
\Ov (\la_i, \la_j):= R_i^* R_j L_j L_i^* = \langle R_i | R_j \rangle \langle L_j | L_i \rangle
$$
where $ L_i, L_j, R_i, R_j $ are the left and right eigenvectors associated to $\la_i$ and $\la_j$ respectively. Overlaps form an hermitian positive-definite matrix than can be shown (deterministically) to have smallest eigenvalue $1$. We give a short proof of this structural fact in the Appendix. \medskip

The diagonal overlap $\Ov (\la_i, \la_i)$ is a real number greater than one, sometimes called the eigenvalue condition number. Considering what has been said about the angle between eigenvectors, it is clear that
$$
\Ov(\la_i, \barla_i) =0,
\qquad
\Ov (\la_i, \la_i)
= \Ov (\barla_i, \barla_i),
$$
and this latter quantity will be referred to as $\Ov_{i,i}$. Eigenvalues being exchangeable, so is the matrix of overlaps; in particular every diagonal overlap is distributed like $\Ov_{1,1}$. The following theorem is stated with respect to scaled eigenvalues, i.e. eigenvalues of $\frac{1}{\sqrt{2N}} G$.

\begin{theorem}[Distribution of diagonal overlaps]\label{diagoverlap} Conditionally on $\Phi(\{\la_1, \dots, \la_N\}) \in \D^{2N}$, the following identity in distribution holds.
$$
\Ov_{1,1} \disteq \prod_{k=2}^N \ll( 1+ \frac{|X_k|^2}{2N |\la_1-\la_k|^2} + \frac{|Y_k|^2}{2N |\la_1- \barla_k|^2 
}\rr)
$$
where all $X_k, Y_k$ variables are i.i.d. standard complex Gaussians. In particular,
$$
\E \ll( \Ov_{1,1} \ | \ \Phi(\{\la_1, \dots, \la_N\}) \rr) = \prod_{k=2}^N \ll( 1+ \frac{1}{2N |\la_1-\la_k|^2} + \frac{1}{2N |\la_1- \barla_k|^2 
}\rr).
$$
\end{theorem}

This theorem is the quaternionic analog of what was established for the complex Ginibre ensemble in \cite{BourgadeDubach, Fyodorov} and for the real Ginibre ensemble in \cite{Fyodorov}.

\begin{proof} We define the partial sums
$$
\Ov_{1,1}^{(2d)} = \sum_{k=1}^{2d} |b_k|^2 = \| {b}^{(2d)} \|^2.
$$
By what has been written above, it is clear that $
\Ov_{1,1}^{(2)}=1$ and $\Ov_{1,1}= \Ov_{1,1}^{(2N)}$. In between, the following recurrence takes place:

\begin{align*}
\Ov_{1,1}^{(2d+2)} & = \Ov_{1,1}^{(2d)} + |b_{2d+1}|^2 + |b_{2d+2}|^2 \\
& = \Ov_{1,1}^{(2d)} + \frac{1}{2N|\la_1-\la_{d+1}|^2} |b^{(2d)}.T_{(2d+1)}|^2 + \frac{1}{2N |\la_1- \barla_{d+1}|^2} |b^{(2d)}.T_{(2d+2)}|^2 \\
&= \Ov_{1,1}^{(2d)} \ll(1 + \frac{1}{2N |\la_1-\la_{d+1}|^2} \frac{|b^{(2d)}.T_{(2d+1)}|^2}{\|b^{(2d)}\|^2} + \frac{1}{2N |\la_1- \barla_{d+1}|^2} \frac{|b^{(2d)}.T_{(2d+2)}|^2}{\|b^{(2d)}\|^2}  \rr) 
\end{align*}

It is enough to notice that
$$
X_d:= \frac{b^{(2d)}.T_{(2d+1)}}{\|b^{(2d)}\|}, \qquad Y_d:= \frac{b^{(2d)}.T_{(2d+2)}}{\|b^{(2d)}\|} 
$$
are i.i.d. Gaussian and independent of $\Ov_{1,1}^{(2d)}$ due to the fact that $T_{(2d)}, T_{(2d+1)}$ are Gaussian vectors independent from the rest of the matrix, and such that $T_{(2d+2)}= \Phi(T_{(2d+1)})$. Indeed, using the facts  (\ref{factsphi}) and (\ref{Schurcolumns}), we see that 
$$
b^{(2d)}.T_{(2d+2)} = b^{(2d)}.\Phi(T_{(2d+1)}) = - \overline{\Phi(b^{(2d)}).T_{2d+1}}
$$
and we finally use the fact that, as $ \langle b^{(2d)} | \Phi(b^{(2d)}) \rangle =0$, the projections of a standard Gaussian vector on them yields independent variables.
\end{proof}

\begin{theorem}[Limit of diagonal overlaps]\label{Gamma4ever} Conditionally on $\{ \la_1 =0 \}$,
$$
\frac{1}{2N} \Ov_{i,i} \distconv \ga_4^{-1}
$$
\end{theorem}

\begin{proof}
The computation is analogous to what was done in \cite{BourgadeDubach} in the complex Ginibre case, using the beta-gamma algebra. Theorem \ref{diagoverlap} and Lemma \ref{Kostlancond} give us the following chain of equalities in distribution, where we used that $\ga_2^{(k)}:= |X_k|^2 + |Y_k|^2$ are i.i.d. $\ga_2$ distributed.
$$
\frac{1}{2N} \Ov_{1,1}
\disteq \frac{1}{2N} \prod_{k=2}^N \ll( 1+ \frac{\ga_2^{(k)}}{\gamma_{2k}}\rr) 
\disteq \frac{1}{2N} \prod_{k=2}^N \beta_{2k, 2}^{-1}
\ \disteq \ \frac{1}{2N} \beta_{4,2N}^{-1} \ \distconv \ \ga_4^{-1}
$$
which concludes the proof.
\end{proof}
Note that the above doesn't only yield a limit, but an exact distribution at every fixed $N$, when conditioning on the specific value $\la_1=0$. \medskip

According to Fyodorov's result \cite{Fyodorov}, the relevant limit for the real Ginibre ensemble is $(2 \ga_{1})^{-1}$. In other words, the limit of $N \Ov_{1,1}^{-1}$ when conditioned on $\{ \la_1 =0 \} $ are $2 \ga_1$, $\ga_2$ and $\frac{1}{2} \ga_4$ for the real, complex, and quaternionic case respectively. \medskip

In the quaternionic case, a natural fact to expect is that this $(\frac{1}{2}\ga_4)^{-1} $ limit holds when conditioning anywhere on the real line. At any nonreal point within the bulk, a different limit is to be expected. The proof of such results would require techniques that have not been yet developed.


\appendix
\setcounter{section}{-1}
\section{Appendix -- Matrix of overlaps and lack of normality}
\renewcommand{\thetheorem}{A.\arabic{theorem}}
\renewcommand{\theequation}{A.\arabic{equation}}

For the convenience of the reader, we gather here a few basic facts and proofs related to the matrix of overlaps. In the following, $G$ is an $N \times N$ matrix with distinct eigenvalues. Choosing an arbitrary reduction $G=P \Delta P^{-1}$ with $\Delta = \rm{Diag}(\la_1, \dots, \la_N)$, we consider the Gram matrix $A=P^*P$, and the matrix of overlaps is defined by:
$$
\Ov = (A_{i,j} A^{-1}_{j,i})_{i,j=1}^N.
$$
In other words, the Hadamard product of $A$ and $(A^t)^{-1}$. If $| R_i \rangle$ is the right-eigenvector associated to $\lambda_i$ (column of $P$) and 
$\langle L_i |$ the left-eigenvector for the same eigenvalue (row of $P^{-1}$), we have
$$
\Ov_{i,j} = \langle R_i | R_j \rangle \langle L_j | L_i \rangle.
$$

\begin{proposition}\label{speck}
For any $i$, $\sum_{j=1}^N \Ov_{i,j} =1. $
In particular, $1 \in \rm{Spec} \Ov$.
\end{proposition}

\begin{proof}
Let us consider the matrix $ M= \sum R_j L_j$. It is such that $MP=P$, so $M=I$, and in particular $ 1= R_i^* M L_i^* = \sum_{j=1}^N \Ov_{i,j}$.
\end{proof}

\begin{definition} We define the Frobenius norm by 
$$
\| M \|_2:= \left( \tr \left( M^* M \right) \right)^{\frac{1}{2}} 
= \left( \sum_{i,j=1}^N |M_{i,j}|^2 \right)^{\frac{1}{2}}
= \left( \sum_{i=1}^N s_i^2 \right)^{\frac{1}{2}}
$$
where $(s_i)_{i=1}^N$ denote the singular values of $M$.
\end{definition}

\begin{proposition}[Frobenius Inequality] If $(\la_i)_{i=1}^N$ denote the eigenvalues of $M$, we have
$$
\| M \|_2^2 \geq \sum_{i=1}^N |\lambda_i|^2
$$
\end{proposition}

\begin{proof}
Let $M=UTU^*$ be the Schur decomposition of the matrix $M$. The quantities involved are invariant under unitary conjugation, and 
$$
\| M \|_2^2 = \sum_{i,j=1}^N |T_{i,j}|^2
= \sum_{i=1}^N |\la_i|^2 + \sum_{i<j} |T_{i,j}|^2
\geq \sum_{i=1}^N |\la_i|^2.
$$
The claim follows.
\end{proof}

We shall call the positive quantity $\Lambda(M)= \| M \|_2^2 - \sum_{i=1}^N |\lambda_i|^2$ the lack of normality of $M$. The next proposition compares asymptotically the lack of normality of complex and quaternionic Ginibre matrices.

\begin{proposition}
If $G_{\C}$ is a scaled complex Ginibre matrix of size $N \times N$, then the lack of normality verifies the following limit theorem:
$$
 \Lambda(G_{\C})-\frac{N-1}{2}
\xrightarrow[N \rightarrow \infty]{d} \mathscr{N}_{\R} \left(0,\frac{1}{2}\right).
$$
In the quaternionic Ginibre case, the limit theorem becomes:
$$ \Lambda(G_{\H})- 2 (N-1)
\xrightarrow[N \rightarrow \infty]{d} \mathscr{N}_{\R} \left(0,4 \right).
$$
\end{proposition}

\begin{proof}
The distribution of the matrix $T$ in the complex Ginibre case is well known (see for instance Appendix A of \cite{Ipsenthesis}). In particular, for any $i<j$, $T_{i,j} \sim \mathscr{N}\left(0,\frac{1}{N} \right)$ and these variables are independent. Therefore,
$$
\Lambda(G_{\C})= \frac{1}{N} \sum_{i<j} N|T_{i,j}|^2.
$$
The last term is a sum of $\frac{N(N-1)}{2}$ i.i.d. random variables with mean $1$ and variance $1$. The usual central limit theorem yields the first result. In the quaternionic case, the same argument yields
$$
\Lambda(G_{\H})= \sum_{i<j} \|T_{i,j}\|_2^2 = \frac{1}{N} \sum_{i<j} 2N(|u_{i,j}|^2+|v_{i,j}|^2).
$$
Now the last term is a sum of $N(N-1)$ i.i.d. random variables with mean $2$ and variance $4$, which yields the second result.
\end{proof}

For any function $g: \C \mapsto \C$, we define
$$
g(G) = P g(\Delta) P^{-1} = P \ \rm{Diag}(g(\lambda_1), \dots, g(\lambda_N)) \  P^{-1} 
$$
which extends naturally the action of polynomials.

\begin{proposition}\label{quadraticform}
If $f,g: \C \mapsto \C$ are two functions, then
$$
\tr \left( f(G)^* g(G) \right)
=
\sum_{i,j} \Ov_{i,j} \overline{f(\lambda_j)} g(\lambda_i).
$$
\end{proposition}

\begin{proof}
As $G= P \Delta P^{-1}$ and $A=P^*P$ we have
$$
\tr \left( f(G)^* g(G) \right) = \tr \left( P^{-*}f(\Delta)^* P^*  P g(G) P^{-1} \right) 
= \tr \left(f(\Delta)^* A g(\Delta) A^{-1}  \right)
= \sum_{i,j} \overline{f(\la_i)} A_{i,j} g(\la_j) A^{-1}_{j,i}
$$
which is the claim, as $\Ov = A_{i,j} A^{-1}_{j,i}$.
\end{proof}

The interest of the above proposition is that it relates the hermitian matrix $\Ov$, seen as a quadratic form, to the initial matrix G.

\begin{proposition}\label{MinEigOv}
The matrix of overlaps is hermitian and positive definite with $\min \rm{Spec} (\Ov) = 1$.
\end{proposition}

\begin{proof}
Hermitianity is clear from the definition. Using Proposition \ref{quadraticform} with $f=g$, and Frobenius' inequality,
$$
\sum_{i,j} \Ov_{i,j} \overline{f(\lambda_j)} f(\lambda_i)
= \tr \ll( f(G)^* f(G) \rr)
\geq 
\sum_{i=1}^N |f(\la_i)|^2.
$$
This holds for any function $f$, and it  follows by classical reduction theory of quadratic forms that $\min \rm{Spec} (\Ov) \geq 1$. Proposition \ref{speck} proves equality.
\end{proof}
A longer but elegant proof of Proposition \ref{MinEigOv}, relying on the properties of the Hadamard product of matrices, can be found in \cite{HornJohnson}, chapter 5.

\section*{Acknowledgment} The author would like to thank Percy Deift and Jesper Ipsen for helpful discussions and references on this topic. \medskip

The work of the author is partially supported by his advisor's NSF grant DMS-1812114.

\end{document}